\documentclass[letterpaper, 10 pt, conference]{ieeeconf}  
\IEEEoverridecommandlockouts                              

  \let\labelindent\relax
\usepackage{amsmath} 
\usepackage{amssymb}  

\usepackage{tabulary}
\usepackage[english]{babel}
\usepackage{blindtext}
\usepackage[makeroom]{cancel}
\usepackage{amsfonts}
\usepackage{subfig,pgfplots}
\usepackage{amsthm}  
\usepackage{amssymb}
\usepackage{graphicx}

\usepackage{pgf,tikz}
\pgfplotsset{ 
  compat=newest, 
   legend style =
  {font=\footnotesize \sffamily},
  label style = {font=\small\sffamily},
every tick label/.append style={font=\small}
  }
\usepackage{tabularx}
\usepackage{float}
\usepackage{cite}
\usepgfplotslibrary{fillbetween}
\usetikzlibrary{shapes.multipart}
\usepackage{enumitem} 
\usepackage{bbm} 

\usepackage[strict]{changepage}

\renewcommand{\eqref}[1]{Eq.~(\ref{#1})}  

\newtheorem{theorem}{Theorem}
\newtheorem{lemma}{Lemma}
\newtheorem{assumption}{Assumption}

\newtheorem{proposition}{Proposition}

\newtheorem{example}{Example}

\newcommand\scalemath[2]{\scalebox{#1}{\mbox{\ensuremath{\displaystyle #2}}}}

\def\h{2.5}
\def\l{3.45}

\title{\LARGE \bf
A mean-field analysis of a network behavioural--epidemic model 
}

\author{Kathinka Frieswijk, Lorenzo Zino, Mengbin Ye, Alessandro Rizzo, and Ming Cao
\thanks{K. Frieswijk, L. Zino, and M. Cao are with the Faculty of Science and Engineering, University of Groningen, Groningen, the Netherlands (\texttt{\{k.frieswijk,lorenzo.zino,m.cao\}@rug.nl}). 
M. Ye is with the Centre for Optimisation and Decision Science, Curtin University, Perth, Australia (\texttt{mengbin.ye@curtin.edu.au}).
A. Rizzo is with the Department of Electronics and Telecommunications, Politecnico di Torino, Torino, Italy, and with the Institute for Invention, Innovation, and Entrepreneurship, New York University Tandon School of Engineering, Brooklyn NY, USA (\texttt{alessandro.rizzo@polito.it}). This work was partially supported by the European Research Council (ERC-CoG-771687), and the Western Australian Government (Premier's Science Fellowship Program).}%
}

\begin{document}

\maketitle
\thispagestyle{empty}

\begin{abstract}
The spread of an epidemic disease and the population’s collective behavioural response are deeply intertwined, influencing each other's evolution. Such a co-evolution typically has been overlooked in mathematical models, limiting their real-world applicability. To address this gap, we propose and analyse a behavioural--epidemic model, in which a susceptible--infected--susceptible epidemic model and an evolutionary game-theoretic decision-making mechanism concerning the use of self-protective measures are coupled. Through a mean-field approach, we characterise the asymptotic behaviour of the system, deriving conditions for global convergence to a disease-free equilibrium and characterising the endemic equilibria of the system and their (local) stability. Interestingly, for a certain range of the model parameters, we prove global convergence to a limit cycle, characterised by periodic epidemic outbreaks.\end{abstract}

\section{Introduction}\label{sec:intro}
Mathematical models of epidemic spreading on networks have been of increasing interest to the systems and control community~\cite{Nowzari2016,Mei2017,Pare2020,zinoreview}. Since 2020, the COVID-19 pandemic has given an extra impetus to such an interest~\cite{giordano2020modelling,DellaRossa2020}. In particular, the ongoing pandemic has highlighted the key role of human behavioural response in shaping the course of an epidemic outbreak and how such a response is deeply intertwined with the epidemic spreading process. Some efforts have been made to incorporate human behaviour into epidemic models~\cite{Funk2010,Wang2015,Sahneh2012,Granell2013,9089218,frieswijk2021time}, in particular, by adding an \textit{alert} state, in which individuals take self-protective measures based on factors such as the awareness of the infection prevalence \cite{Sahneh2012,frieswijk2021time}, communication with neighbours \cite{Granell2013},  awareness campaigns \cite{9089218,frieswijk2021time}, or by incorporating opinion dynamics mechanisms~\cite{Peng2021,She2022}. While these models proved useful in capturing some key aspects of real-world epidemics, their inherent oversimplification of the evolving nature of human behaviour limits their practical applicability.

Recently, evolutionary game theory has emerged as a powerful framework to develop realistic behavioural--epidemic models~\cite{Huang2022,Hota2019,Khazaei2021,Elokda2021,Martins2022,Satapathi2022}. In~\cite{ye2021game}, a novel game-theoretic paradigm was proposed, in which human decision making and epidemics co-evolve on a two-layered network, with the decision making influenced by a range of factors such as social influence, interventions, risk perception, and immediate and accumulated costs of using protection. However, except for the approximation of the epidemic threshold,~\cite{ye2021game} relies only on numerical simulations, which suggest that the behavioural--epidemic model can reproduce a wide range of behaviours, including eradication of the disease, convergence to endemic equilibria, or periodic oscillations and multiple epidemic waves.

In this letter, we expand on~\cite{ye2021game} to provide an analytical treatment of the long-term behaviour of a game-theoretical behavioural--epidemic model. To this aim, we propose a continuous-time implementation of the framework proposed in~\cite{ye2021game}, combined with a susceptible--infected--susceptible epidemic model. Through a mean-field approach~\cite{VanMieghem2009}, we derive analytical results on the asymptotic behaviour of the system. After having established the epidemic threshold, we analyse the behaviour of the system below and above such a threshold. Below the threshold, we prove global convergence to a disease-free equilibrium (DFE). Above the threshold, we characterise the endemic equilibria (EEs) of the system and their local stability properties. Furthermore, we derive conditions under which the system undergoes periodic oscillations with multiple waves, converging to a limit cycle. Finally, numerical simulations suggest that the locally exponentially stable equilibria are also globally stable, paving the way for future research towards extending our theoretical findings.

\section{Model}\label{sectionmodel}


\emph{Notation:} The set of real, real nonnegative, and strictly positive real numbers is denoted by $\mathbb{R}$, $\mathbb{R}_{\ge 0}$, and $\mathbb{R}_{> 0}$, respectively. 
We say that an event $E$ is triggered by a \emph{Poisson clock} with (possibly time-varying) rate $q_E(t)$, if $\lim_{\Delta t \searrow 0} {\mathbb{P}\big[E\text{ occurs during }(t,t+\Delta t)\big]}/{\Delta t} =q_E (t)$.

\subsection{Population and Network Model}\label{network}

We consider a population of $n$ individuals $\mathcal{V}=\{1,\hdots,n\}$. Each individual $i \in \mathcal{V}$ is characterised by a two-dimensional state $(x_i(t), y_i(t))$, reflecting their \emph{behavioural state} $x_i(t) \in \{0,1\}$ and \emph{health state} $y_i(t) \in \{S,I \}$, at time $t \in \mathbb{R}_{\ge 0}$. In particular, an individual $i \in \mathcal{V}$ either chooses to use self-protective measures ($x_i(t) =1$) at time $t$, thereby preventing any possible contraction of the disease, or to not employ them ($x_i(t) =0$); simultaneously, the individual can have two different health states: $y_i(t) = I$ if $i$ is infected, and $y_i(t) = S$ if $i$ is healthy and susceptible to the infection. 

Each individual is represented by a node in a two-layer temporal network $\mathcal{G}(t) = (\mathcal{V}, \mathcal{E}_{\text{I}}, \mathcal{E}_{\text{C}}(t))$, illustrated in Fig.~\ref{fig:network}. The \em influence layer \em captures social influence on the individual's decision-making process through the (possibly directed) link set $\mathcal{E}_{\text{I}}$, whereby node $j$ is an (out)-neighbour of $i$ ($(i,j)  \in \mathcal{E}_{\text{I}}$) if and only if (iff)  $j$ can influence $i$'s behaviour. The set of neighbours of $i$ is denoted by $\mathcal N_i : =  \{j \in \mathcal{V} \ :  \ (i,j) \in \mathcal{E}_{\text{I}}  \}$, with size $d_i:=|\mathcal N_i|$. Since the spreading of a disease typically evolves much faster than social ties do, we assume that the influence layer is time-invariant. 

Disease transmission from an infectious to a susceptible individual occurs through interactions in close physical proximity, henceforth denoted by \emph{contacts}, modelled by the \textit{contact layer} $\mathcal{E}_{\text{C}}(t)$, where  $\{i,j\}  \in \mathcal{E}_{\text{C}}(t)$ iff $i$ and $j$ have a contact at time $t \in \mathbb{R}_{\ge0}$. We assume that contacts are generated according to a continuous-time activity-driven network~\cite{Zino2016}, in which each individual $i \in \mathcal{V}$ is assigned an \emph{activity rate} $a_i \in \mathbb{R}_{>0}$, which captures the level of physical activity of individual $i$. Then, $i$ activates if triggered by a Poisson clock with rate $a_i$ and, once active, generates a contact with another individual, selected uniformly at random from $\mathcal{V}\setminus\{i\}$. 

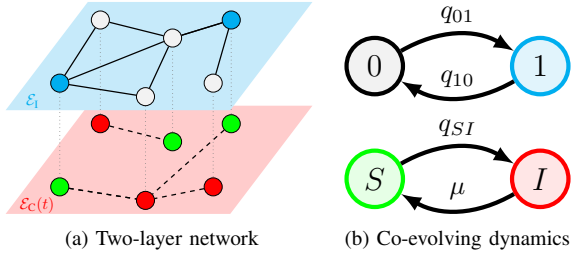
\begin{figure}
    \centering
    \subfloat[Two-layer network]{\scalebox{0.6}{\begin {tikzpicture}

\usetikzlibrary{shapes.geometric}

\tikzstyle{trap}=[trapezium, trapezium stretches=true, draw=none, minimum width=6.8cm,  minimum height=2.4cm, trapezium left angle=60, trapezium right angle=120]
\tikzstyle{peers}=[draw,circle, text=black,  fill=white,inner sep=0pt, minimum size=.4cm]
\node[trap,fill=red!20]
    at (1.3,2.1) {};
\node[trap,fill=cyan!20]
    at (1.3,4.4) {};

  \node at (-1.5,3.4) {\color{cyan}$\mathcal E_{\text{I}}$};   
 \node at (-1.4,1.1) {\color{red}$\mathcal E_{\text{C}}(t)$};

\node[peers,fill=cyan] (1) at (-0.9,3.8) {};
\node[peers,fill=gray!10] (2) at (1,3.5) {};
\node[peers,fill=gray!10] (3) at (2.5,3.8) {};

\node[peers,fill=gray!10] (4) at (0,5.2) {};
\node[peers,fill=gray!10] (5) at (1.6,4.8) {};
\node[peers,fill=cyan] (6) at (2.9,5.2) {};

\node[peers,fill=green] (11) at (-0.9,1.5) {};
\node[peers,fill=red] (12) at (1,1.2) {};
\node[peers,fill=red] (13) at (2.5,1.5) {};

\node[peers,fill=red] (14) at (0,2.9) {};
\node[peers,fill=green] (15) at (1.6,2.5) {};
\node[peers,fill=green] (16) at (2.9,2.9) {};

\foreach \i/\j in {1/2,1/4,1/5,4/5,6/3,6/5,2/5,5/6}
{\draw[thick] (\i) edge  (\j) ;}

\foreach \i/\j in {11/12,11/12,16/12,13/12,14/15}
{\draw[dashed,thick] (\i) edge (\j) ;}

\foreach \i/\j in {11/1,2/12,3/13,4/14,5/15,6/16}
{\draw[-,dotted, >=latex,gray] (\i) edge (\j);}

\end{tikzpicture}}\label{fig:network}}\quad\subfloat[Co-evolving dynamics]{\begin{tikzpicture} \node[draw=red, fill=red!10,circle, ultra thick,,minimum size=0.75cm] (I1) at (3.5,1.5) {\large $I$};
\node[draw=green, fill=green!10,circle, ultra thick,,minimum size=0.75cm]  (S1) at (1.3,1.50) {\large $S$};
\path [->,>=latex,ultra thick]  (S1) edge[bend left =30]   node [above] {{$q_{SI}$}} (I1);
\path [->,>=latex,ultra thick]  (I1) edge[bend left =30]   node [above] {{$\mu$}} (S1);

\node[draw=black, fill=gray!10,circle, ultra thick,,minimum size=0.75cm] (0) at (1.3,3) {\large $0$};
\node[draw=cyan, fill=cyan!10,circle, ultra thick,,minimum size=0.75cm]  (1) at (3.5,3) {\large $1$};
\path [->,>=latex,ultra thick]  (0) edge[bend left =30]   node [above] {{$q_{01}$}} (1);
\path [->,>=latex,ultra thick]  (1) edge[bend left =30]   node [above] {{$q_{10}$}} (0);
\end{tikzpicture}\label{fig:dynamics}}
    \caption{Illustration of the network model and dynamics. }
    \label{fig:schematic} 
\end{figure}

\subsection{Behavioural--Epidemic Model}\label{sectionepidemicmodel}

In the behavioural--epidemic framework proposed in~\cite{ye2021game}, each individual $i\in \mathcal{V}$ decides whether to adopt self-protective measures according to an evolutionary game-theoretic mechanism~\cite{hofbauer1998evolutionary}, depending on social influence, risk perception, costs for adopting self-protective measures, frustration, and government policy interventions. Here, we propose a simplified decision-making mechanism in which the last two factors are omitted. Such a simplification allows the reduction of the number of parameters involved in the system, simplifying its analysis and the presentation of the results, without restricting the broad range of possible emergent behaviours, as we shall demonstrate in this letter.

To capture these factors, we introduce the payoff function
\begin{subequations}\label{payoff}
 \begin{align}
       \pi_1^{(i)} (t) &=\displaystyle\frac{1}{d_i} \sum\nolimits_{j\in\mathcal N_i} x_j(t)  + \zeta \bar{y}(t)\,,\label{payoff1}
       \intertext{which captures the payoff for adopting self-protective measures ($x_i=1$),  where $\bar{y}(t) :=  \tfrac{1}{n}\big| \{i \in \mathcal{V} \ :  \ y_i(t) =  I \} \big|$ denotes the infection prevalence at time $t$; and} 
    \pi_0^{(i)} (t)&=\displaystyle\frac{1}{d_i} \sum\nolimits_{j\in\mathcal N_i} \big(1-x_j(t)\big)  +c\,,\label{payoff0}
  \end{align}
\end{subequations}
which captures the payoff associated with not adopting self-protections. The first term, present in both formulae, represents \textit{social influence}: the more neighbours of $i$ adopt a certain action, the higher the payoff for the corresponding action. The term $\zeta \bar y(t)$, with $\zeta\in\mathbb R_{\geq 0}$, increases the payoff for adopting self-protections as the infection prevalence grows, capturing the \textit{risk perception}. Here, we assume that people react in a linear fashion in response to the information they receive on the infection prevalence $\zeta \bar y(t)$, but more complex and nonlinear terms may be considered. Finally, the constant $c\in\mathbb R_{\geq 0}$ represents the psychological, social, and economical \textit{cost} per unit-time associated with the adoption of self-protections, thereby increasing the payoff for not adopting self-protections.

Individuals change their behaviour following a stochastic implementation of the classical \textit{imitation dynamics} mechanism, which is often used in evolutionary game theory~\cite{hofbauer1998evolutionary,Como2021}, in which they imitate their peers triggered by Poisson clocks with rate equal to their corresponding payoff functions. Specifically, an individual $i$ who is not adopting self-protective measures at time $t$ (i.e.\ $x_i(t)=0$) will adopt them if triggered by a Poisson clock with rate \begin{subequations}\label{behaviourevolution}
 \begin{align}
      q_{01}^{(i)}(t)&=\frac{1}{d_i}\sum\nolimits_{j\in \mathcal N_i}x_j(t)\pi_1^{(j)}(t)\,,\label{behaviourevolution_01}
 \intertext{and an individual $i$ who is adopting them (i.e. $x_i(t)=1$) will stop if triggered by a Poisson clock with rate}  q_{10}^{(i)}(t)&=\frac{1}{d_i}\sum\nolimits_{j\in \mathcal N_i}\big(1-x_j(t)\big)\pi_0^{(j)}(t)\,.\label{behaviourevolution_10}
  \end{align}
\end{subequations}

Simultaneously, if a susceptible individual $i$ ($y_i(t) = S$) who does not use protective measures ($x_i(t) =0$) has a physical encounter with an infected individual $k$ ($y_k(t) =I$), then $i$ becomes infected with \textit{per-contact infection probability} $\lambda\in (0,1]$. We assume that self-protective measures are $100\%$ effective in preventing contagion. Hence, if individual $i$ employs protections at time $t$ ($x_i(t) =1$), then they cannot be infected at time $t$. Following~\cite{Zino2017}, we compute that if $i$ is susceptible at time $t$ ($y_i(t) = S$), then $i$ will become infected if triggered by a Poisson clock with rate 
\begin{equation}\label{infectionrate}
         q_{SI}^{(i)}(t)= \dfrac{\lambda(1-x_i(t))}{n-1}\bigg(na_i\bar y(t) +  \sum_{j\in \mathcal{V}:  y_j(t) = I} a_j\bigg)\,,
\end{equation}
where the first term in the parentheses accounts for the contact initiated by $i$ with infected individuals, and the second accounts for contacts initiated by infected individuals who interact with $i$. If the disease can be transmitted only in one direction, then only the corresponding term should be considered in \eqref{infectionrate}. Note that if individual $i$ employs protection at time $t$, then $q_{SI}^{(i)}(t)=0$. An infected individual $i$ $\left(y_i(t) = I\right)$ spontaneously recovers, if triggered by a Poisson clock with node-independent and time-invariant rate $\mu\in \mathbb{R}_{>0}$. All the state transitions and rates are shown in Fig.~\ref{fig:dynamics}.

\section{Mean-Field Dynamics}\label{sectionmf}
The evolution of the state of each individual $(x_i(t),y_i(t))$, $i\in\mathcal V$, is determined by independent Poisson clocks. Hence, the state of the system follows a Markov process on a state space with size growing exponentially with the population size $n$, making its direct analysis unfeasible. We employ a mean-field relaxation of the stochastic process to derive analytical insight, following the $n$-intertwined mean-field approach described in~\cite{VanMieghem2009}. Specifically, we define and study for each individual $i \in \mathcal{V}$ the probabilities of adopting protective behaviours $p_x^{(i)}(t)  : = \mathbb{P}\left[ x_i(t) = 1 \right]$ and of being infected  $p_y^{(i)}(t) : = \mathbb{P}\left[ y_i(t) = I\right]$, which evolve according to
\begin{subequations}\label{meanfield_evo}\begin{align}
      \dot p_x^{(i)}  &=(1-p_x^{(i)})q_{01}^{(i)}-p_x^{(i)}q_{10}^{(i)}\,,\label{meanfield_evox}\\ 
      \dot p_y^{(i)}&=(1-p_y^{(i)})q_{SI}^{(i)}-p_y^{(i)}\mu\label{meanfield_evoy}\,.
      \end{align}
\end{subequations}

Also, we introduce the macroscopic variables
\begin{equation}\label{macro}
  x(t):= \dfrac{1}{n}\sum\nolimits_{i \in \mathcal{V}} p_x^{(i)}(t)\,,\quad  y(t) : = \dfrac{1}{n}\sum\nolimits_{i \in \mathcal{V}} p_y^{(i)}(t)\,,
\end{equation}
which are the average probability that a randomly selected individual is adopting protections and is infected at time $t$, respectively. Let $\bar x(t):=\frac1n\sum_{i\in\mathcal V} x_i(t)$ denote the fraction of adopters of self-protection in the population at time $t$.
In the limit of large-scale populations, $n \to \infty$, the central limit theorem ensures that $\bar x(t)$ and $\bar y(t)$ converge to $x(t)$ and $y(t)$, respectively. Hence, the macroscopic variables in \eqref{macro} approximate with arbitrary accuracy the fraction of adopters of self-protective measures and the epidemic prevalence, for any finite-time horizon~\cite{limitmarkov}.

In the rest of this letter, we will make the following simplifying assumption.  

\begin{assumption}\label{assumptionsimple}
We assume that a) the influence layer is complete, i.e.\ $\mathcal N_i=\mathcal V,\,\forall\,i\in\mathcal V$;
b) individuals have homogeneous activity, i.e.\ $a_i = \alpha \in \mathbb{R}_{>0}$, $\forall\,i \in \mathcal{V}$; and c) individuals have the same initial probability of adopting protections, i.e.\ $p_x^{(i)}(0) = p_x(0) \in [0,1], \forall\,i \in \mathcal{V}$.
\end{assumption}
Under item a) of Assumption \ref{assumptionsimple}, \eqref{payoff}  reduces to $\pi_{1} (t) = x(t)  + \zeta  y (t)$ and $\pi_{0} (t)=1- x(t)+c$, where we have dropped the index $i$ since the payoffs are uniform across the population. Under Assumption~\ref{assumptionsimple}, we  derive a planar system that governs the mean-field evolution of the macroscopic variables (proof in Appendix~\ref{app:proofsystem}) and rigorously analyse it.

\begin{proposition}\label{prop:system}In the limit of large-scale populations $n\to\infty$ and under Assumption~\ref{assumptionsimple}, the two macroscopic quantities in \eqref{macro} evolve according to the following planar system:
\begin{equation}\label{simplesystem}
   \begin{array}{lll}
       \dot{x} & =& x(1-x) (2x+\zeta y-1 -c)\,,\\
    \dot{y} & =& 2\alpha \lambda y(1-x)(1-y) - \mu y\,.
\end{array} 
\end{equation}
\end{proposition}

The following result guarantees that \eqref{simplesystem} is always well-defined, i.e.\ that the variables $x$ and $y$, which represent fractions of the population, remain within $[0,1]\times[0,1]$.
\begin{lemma}\label{invariantdomain}
The domain $[0,1]\times[0,1]$ is positively invariant for \eqref{simplesystem}.
\end{lemma}
\begin{proof}
The domain $[0,1]\times[0,1]$ is compact and convex and the vector field in \eqref{simplesystem} is Lipschitz-continuous. Hence, Nagumo's Theorem can be applied (see~\cite{Blanchini1999}). We are left with checking the direction of the vector field at the boundaries of the domain. We observe that $\dot x=0$ for $x=0$ and $x=1$, while $\dot y=0$ for $y=0$ and $\dot y<0$ for $y=1$, implying that any trajectory such that $(x(0),y(0))\in[0,1]\times[0,1]$ has $(x(t),y(t))\in[0,1]\times[0,1]$ for any $t\geq 0$.  
\end{proof}

In the following, we will make some realistic assumptions on the model parameters. In particular, we want to guarantee that the use of self-protective measures is always preferred when the entire population is infected, while their use is disfavoured in the absence of a disease. To guarantee this, we need to enforce in \eqref{payoff} that, for any $x \in [0,1]$, $\pi_1(t)<\pi_0(t)$, if $y=0$, and  $\pi_0(t)<\pi_1(t)$,  if $y=1$. These conditions are satisfied by making the following assumption. 
\begin{assumption}\label{as:c}
We assume that $c>1$ and $\zeta>c+1$.
\end{assumption}

\begin{figure*}
\centering
\subfloat[$\zeta=5$]{\label{fig:gamma1}\includegraphics[width=0.29\linewidth]{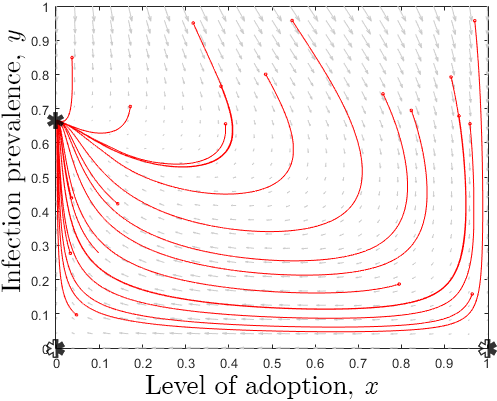}}\qquad
\subfloat[$\zeta=8$]{\label{fig:gamma2}\includegraphics[width=0.29\linewidth]{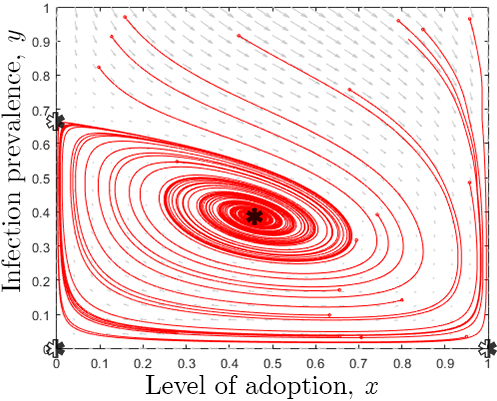}} \qquad
\subfloat[$\zeta=9.5$]{\label{fig:gamma3}\includegraphics[width=0.29\linewidth]{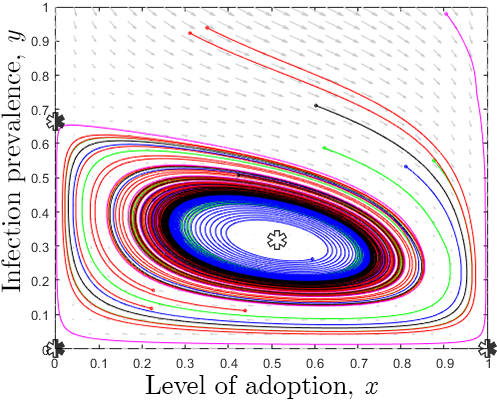}}\vspace{-0.15cm}
\caption{Simulations of \eqref{simplesystem} for different values of the risk perception parameter (in the captions) for Example~\ref{ex}. 
Stable equilibria, saddle points and unstable equilibria are marked with a black, black-white and white asterisk, respectively.}\label{zeta}
\end{figure*}
\section{Main Results}\label{sectionresults}
We study the asymptotic behaviour and the equilibria characteristics of the behavioural--epidemic model using the mean-field system in \eqref{simplesystem}. The following lemma characterises the equilibria of the system in \eqref{simplesystem} and their local stability properties. Its proof can be found in Appendix \ref{app:proofprop1}.

\begin{lemma}\label{propequi}
Under Assumption \ref{as:c}, \eqref{simplesystem} has at most five equilibria: three on the boundary of $[0,1]\times[0,1]$, two in the interior. The three equilibria on the boundary are:
\begin{enumerate}
        \item[i)] the DFE $(0,0)$, which is locally asymptotically stable if  $\lambda \leq \tfrac{\mu}{2\alpha}$ (with exponential stability if strict inequality holds), and a saddle point if $\lambda > \tfrac{\mu}{2\alpha}$;
        \item[ii)] the DFE $(1,0)$, which is a saddle point;
        \item[iii)] the protection-free EE $(0,1-\tfrac{\mu}{2\alpha \lambda })$, which exists iff $\lambda> \tfrac{\mu}{2\alpha}$. When it exists, it is locally asymptotically stable if $\zeta\leq  \tfrac{2\alpha \lambda(1+c)}{2\alpha \lambda-\mu}$ (with exponential stability if strict inequality holds) and a saddle point if $\zeta> \frac{2\alpha \lambda(1+c)}{2\alpha \lambda-\mu}$. 
    \end{enumerate}
Next, define \begin{equation*}
 \scalemath{0.89}{   \beta_{\pm} := \dfrac{1}{4} \left[c+3-\zeta   \pm  \sqrt{(c+3-\zeta)^2 + 8\left[\zeta\left(1-\tfrac{\mu}{2\alpha \lambda}\right) -1-c\right] }\right]}. 
\end{equation*}
The two EEs in the interior are:
\begin{enumerate}
  \item[iv)] $\scalemath{0.89}{\big(\beta_{+},1-\tfrac{\mu}{2\alpha \lambda(1-\beta_{+})}\big)}$, which exists iff $ \scalemath{0.89}{\lambda> \tfrac{\mu}{2\alpha(1-\beta_{+})}}$ and one of the following conditions is satisfied: a) $\scalemath{0.855}{c-1 + \tfrac{2 \mu}{\alpha \lambda} +\sqrt{\tfrac{4 \mu}{\alpha \lambda} \left(c-1 + \tfrac{ \mu}{ \alpha \lambda} \right)} \le \zeta < c+3}$, where necessarily $\scalemath{0.855}{ c < \tfrac{4 \alpha \lambda}{\mu}-3}$;
or $\scalemath{0.855}{\zeta \ge c+3}$ and  $\scalemath{0.855}{ \zeta> \tfrac{2\alpha \lambda(1+c)}{2\alpha \lambda-\mu}}$. 
  If it exists, it is locally exponentially stable if $\tfrac{4 \alpha \lambda}{\zeta} (1-\beta_+)^2 < \mu < 2(1-\beta_+)(\alpha \lambda- \beta_+)$, and unstable if $\mu<\tfrac{4 \alpha \lambda}{\zeta} (1-\beta_+)^2$ or $\mu > 2(1-\beta_+)(\alpha \lambda- \beta_+)$.
  \item[v)] $\scalemath{0.855}{\big(\beta_{-},1-\tfrac{\mu}{2\alpha \lambda(1-\beta_{-})}\big)}$, which exists iff $ \scalemath{0.89}{\lambda> \tfrac{\mu}{2\alpha(1-\beta_{-})}}$ and $\scalemath{0.855}{c-1 + \tfrac{2 \mu}{\alpha \lambda} +\sqrt{\tfrac{4 \mu}{\alpha \lambda} \left(c-1 + \tfrac{ \mu}{ \alpha \lambda} \right)} \le \zeta   < \min\left\{ c+3, \tfrac{2\alpha \lambda(c+1)}{2\alpha\lambda-\mu} \right\}}$, where necessarily $\scalemath{0.855}{c < \tfrac{4 \alpha \lambda}{\mu}-3}$.   If it exists, it is locally exponentially stable if $\tfrac{4 \alpha \lambda}{\zeta} (1-\beta_-)^2 < \mu < 2(1-\beta_-)(\alpha \lambda- \beta_-)$, and unstable if $\mu<\tfrac{4 \alpha \lambda}{\zeta} (1-\beta_-)^2$ or $\mu > 2(1-\beta_-)(\alpha \lambda- \beta_-)$. 
\end{enumerate}
\end{lemma}

Lemma \ref{propequi} leads to the establishment of the \textit{epidemic threshold} for the system in \eqref{simplesystem}, i.e.\ the conditions under which the system converges to one of the DFEs.

\begin{theorem}\label{prop1}
Assume that Assumption \ref{as:c} holds and $\lambda\leq   \tfrac{\mu}{2\alpha}$. Then, if $x(0)<1$, the system in \eqref{simplesystem} converges to the DFE $(0,0)$; otherwise, it converges to the DFE  $(1,0)$. 
\end{theorem}
\begin{proof}
Let $f(y)=2\alpha\lambda y(1-y)-\mu y$. By~\cite[Lemma 4.1]{Mei2017}, the solution of $\dot z=f(z)$ converges to $z=0$ if $\lambda\leq  \tfrac{\mu}{2\alpha}$. As a consequence of the nonnegativity of $y(t)$ (Lemma~\ref{invariantdomain}) and the fact that $\dot y\leq f(y)$, $y(t)$ converges to $0$. If $x(0)=1$, convergence to $(1,0)$ is straightforward, since $x=1$ is an invariant manifold. Otherwise, since $y(t)\to 0$ and $c>1$, there exists a time $\bar t>0$ such that $y(t)\leq\frac{c-1}{2\zeta}$, for any $t\geq\bar t$. From \eqref{simplesystem}, we observe that for any $\scalemath{1}{t>\bar t}$, $\scalemath{1}{\dot x = x(1-x)(2x+\zeta y-1-c)<-x(1-x)(c-1)}/2$, which yields the claim.
\end{proof}

Theorem~\ref{prop1} fully characterises the behaviour of the system when the disease is not highly infectious, i.e.\ $\lambda\leq \tfrac{\mu}{2\alpha}$. In the following, we will consider the opposite scenario  $\lambda> \tfrac{\mu}{2\alpha}$. We also assume $c\ge \tfrac{4 \alpha \lambda}{\mu}-3$, ensuring that the unique endemic equilibrium only exists for a high enough level of risk perception $\zeta$. Under an upper bound of $c$, which depends on the other model parameters, Lemma \ref{propequi} reduces to the following proposition, which clearly illustrates the role of the risk perception $\zeta$. The proof is reported in Appendix \ref{app:proofprop2}.

\begin{proposition}\label{prop2}
Let $\lambda> \tfrac{\mu}{2\alpha}$ and $\tfrac{4 \alpha \lambda}{\mu}-3\le c<\tfrac{32 \alpha \lambda}{5 \mu}-3$. Under Assumption~\ref{as:c}, the following hold:
\begin{enumerate}
    \item[i)] if $\zeta < \tfrac{2 \alpha \lambda(1+c)}{2\alpha \lambda-\mu}$, then \eqref{simplesystem} has three equilibria: the DFEs $(0,0)$ and $(1,0)$, which are saddle points, and the (locally) exponentially stable EE $(0,1-\tfrac{\mu}{2\alpha \lambda })$;
    \item[ii)] if $\zeta > \tfrac{2 \alpha \lambda(1+c)}{2\alpha \lambda-\mu}$, then \eqref{simplesystem} has four equilibria: three saddle points---the DFEs $(0,0)$, $(1,0)$, and the EE $(0,1-\tfrac{\mu}{2\alpha \lambda })$---and the interior EE $(\beta_{+},1-\tfrac{\mu}{2\alpha \lambda(1-\beta_{+})})$, which is (locally) exponentially stable iff $\zeta$ satisfies all of the following three conditions: a) $ \scalemath{0.8}{  \zeta > c-1 + \tfrac{25\mu}{8\alpha \lambda} + \tfrac{5}{2} \sqrt{\tfrac{\mu}{\alpha \lambda} (c-1+ \tfrac{25\mu}{16\alpha \lambda})} }$, 
b) $ \scalemath{0.8}{  \zeta > \tfrac{\alpha \lambda}{2 \alpha \lambda- \mu} ((c+1) [1 - \sqrt{(\alpha \lambda -1)^2 +2 \mu}] + \alpha \lambda(c-3) +2 \mu )}$, and c) $ \scalemath{0.8}{  \zeta< \tfrac{\alpha \lambda}{2 \alpha \lambda- \mu} ((c+1) \big[1 + \sqrt{(\alpha \lambda -1)^2 +2 \mu}] + \alpha \lambda(c-3) +2 \mu )}$. 
\end{enumerate}
\end{proposition}

Proposition~\ref{prop2} focuses on local stability and instability of endemic equilibria. In the following result we establish  sufficient conditions under which sustained oscillations with periodic epidemic waves occur, with proof in Appendix \ref{app:prooftheo1}.

\begin{theorem}\label{theo1} Let Assumption~\ref{as:c} hold,  $\lambda> \tfrac{\mu}{2\alpha}$, and $\tfrac{4 \alpha \lambda}{\mu}-3\le c<\tfrac{32 \alpha \lambda}{5 \mu}-3$. Furthermore, assume $\scalemath{0.92}{\zeta>\tfrac{2\alpha \lambda(1+c)}{2\alpha \lambda-\mu}}$, $\scalemath{0.92}{\zeta > c-1 + \tfrac{25\mu}{8\alpha \lambda} + \tfrac{5}{2} \sqrt{\tfrac{\mu}{\alpha \lambda} (c-1+ \tfrac{25\mu}{16\alpha \lambda})} }$, and $\scalemath{0.92}\zeta>  \tfrac{\alpha \lambda}{2 \alpha \lambda- \mu} [(c+1) [1 + \sqrt{(\alpha \lambda -1)^2 +2 \mu}] + \alpha \lambda(c-3) +2 \mu  ]$. 
If the initial condition $x(0),y(0)$ is in the interior of the domain $[0,1]\times[0,1]$, then the system in \eqref{simplesystem} converges to a periodic solution, within the domain $[0,1] \times [0, 1-\frac{\mu}{2\lambda\alpha}]$.
\end{theorem}

We conclude the section by presenting an example.

\begin{example}\label{ex}
Let $c=3$, $\alpha=3$, $\lambda=0.5$, and $\mu =1$. In Fig.~\ref{fig:gamma1}, we set the risk perception to $\zeta=5$, which satisfies the conditions in item i) of Proposition~\ref{prop2}. Hence, the only (locally) stable equilibrium of \eqref{simplesystem} is the EE $(0,1-\tfrac{\mu}{2\alpha \lambda })$. Simulations suggest that all trajectories converge to it. 
In Fig.~\ref{fig:gamma2}, we increase the risk perception to $\zeta=8$, which satisfies all the conditions in item ii) of Proposition~\ref{prop2}. Hence, the interior EE is locally stable: all trajectories converge to it, suggesting globally stability. 
Finally, we set $\zeta=9.5$, which satisfies the conditions of Theorem~\ref{theo1}. Consistently, all the trajectories in Fig.~\ref{fig:gamma3} converge to a limit cycle.
\end{example}

\section{Conclusion}\label{sec:con}
We studied a behavioural--epidemic model in which human behaviour and epidemics co-evolve in a mutually influencing manner. Employing a mean-field approach, we painted an extensive picture of the system behaviour, including a stability analysis of the equilibria and the expression of the epidemic threshold. Furthermore, we explored the role of risk perception in the occurrence of periodic oscillations and established conditions for global convergence to such a periodic solution. 

Our promising results pave the way for several avenues of future research. First, the numerical findings suggest that our local stability results might be extended towards obtaining global results. Second, interventions may be incorporated, towards designing control policies to favour a collective behavioural response and mitigate an epidemic outbreak.Third, our theoretical analysis relies on the simplifying Assumption~\ref{assumptionsimple}. Efforts should be placed towards extending our theoretical findings to more general scenarios, including non-trivial directed networks. Finally, further factors should be incorporated into the model, including limited effectiveness of self-protections, accumulation of socio-economic fatigue, and nonlinear terms to capture more complex risk perception (as in~\cite{ye2021game}). This will be key for real-world applications.

\appendix

\subsection{Proof of Proposition~\ref{prop:system}}\label{app:proofsystem}
In the limit $n \to \infty$,  \eqref{meanfield_evox} reduces to 
$
      \dot p_x^{(i)}  =(1-p_x^{(i)}) x(x+ \zeta y) -    p_x^{(i)}(1- x)(1- x+c)$. Similarly, using item b) of Assumption~\ref{assumptionsimple} and \eqref{infectionrate}, we write \eqref{meanfield_evoy} as $
      \dot p_y^{(i)}  =\lambda(1-p_x^{(i)})(1-p_y^{(i)})2\alpha\frac{n}{n-1} y -\mu p_y^{(i)}
$. 
Next, observe that, under item c) of Assumption~\ref{assumptionsimple}, $p_x^{(i)}(t)$ is the same  $\forall\,i\in\mathcal V$, so we drop the index $i$ and write  $p_x^{(i)}(t)=\bar x(t)=x(t)$, where the last equality holds for $n\to \infty$. Finally, we combine the expressions above and \eqref{macro}, to derive \eqref{simplesystem}. 

\subsection{Proof of Lemma~\ref{propequi}}\label{app:proofprop1}
Solving $ \dot{y} =0$ yields $y=0$ or $(1-x)(1-y) = \tfrac{\mu}{2\alpha \lambda}$, where for the latter, $\tfrac{\mu}{2\alpha \lambda}>0$ necessarily requires $x<1$ and $y<1$ as conditions. Here,  $(1-x)(1-y) = \tfrac{\mu}{2\alpha \lambda}$ with $x<1$ and $y<1$ can be written as $y=1-\tfrac{\mu}{2\alpha \lambda(1-x)}$. If $y=0$, then the solutions to $ 0= \dot{x} = x(1-x) (2x-1 -c)$ in the domain $[0,1]$ are given by $x=0$ and  $x=1$. Thus, the only DFEs are $(0,0)$ and $(1,0)$. Next, let us consider equilibria $(x,y)$ with $y=1-\tfrac{\mu}{2\alpha \lambda(1-x)}\in (0,1)$, $x \in [0,1)$, and let $\mu < 2\alpha \lambda (1-x)$ for existence of the equilibria.
Substituting $y=1-\tfrac{\mu}{2\alpha \lambda(1-x)}$ in $  \dot{x} =0$ yields $x(1-x) (2x+\zeta(1-\tfrac{\mu}{2\alpha \lambda(1-x)})-1 -c)=0$, which for $x <1$ has solutions $x=0$, and $\scalemath{0.89}{\beta_{\pm} = \tfrac{1}{4} (c+3-\zeta \pm \sqrt{(c+3-\zeta)^2 + 8[\zeta(1-\tfrac{\mu}{2\alpha \lambda}) -1-c] })}$. The EEs are given by  $(0,1-\tfrac{\mu}{2\alpha \lambda})$ and $(\beta_{\pm},1-\tfrac{\mu}{2\alpha \lambda(1-\beta_{\pm})})$. First, we investigate for which values of $\zeta$ we have  $(c+3-\zeta)^2  + 8[\zeta\left(1-\tfrac{\mu}{2\alpha \lambda}\right) -1-c]= \zeta^2 +2 \zeta ( 1- c - \tfrac{2 \lambda}{\alpha \lambda}) + (c-1)^2 \ge 0$,
which is a necessary condition for $\beta_{\pm} \in \mathbb{R}$. Note that the roots of $\scalemath{1}{\zeta^2 +2 \zeta ( 1- c - \tfrac{2 \lambda}{\alpha \lambda}) + (c-1)^2}$ are $\zeta_{\pm} = c -1 + \tfrac{2 \mu}{\alpha \lambda} \pm \sqrt{\tfrac{4 \mu}{\alpha \lambda} (c-1 + \tfrac{ \mu}{\alpha \lambda})} $. Since $ \zeta_{-}< c-1$,  $\beta_{\pm} \in \mathbb{R}$ iff $\zeta \ge \zeta_{+} $. Next, we study for which values of $\zeta$ and $c$ we have $\beta_{\pm} \in (0,1)$, while assuming that the previously identified conditions necessary for $\beta_{\pm} \in \mathbb{R}$ hold. We start with $\beta_+$. If $\zeta <c+3$, then $\beta_+ >0$. If $\zeta \ge c+3$, then $\beta_+ >0$ iff $\zeta >\tfrac{2\alpha \lambda(1+c)}{2\alpha \lambda-\mu}$, for which $\zeta(1-\tfrac{\mu}{2\alpha \lambda}) -1-c > 0$.
Observe that $\beta_+ <1$ iff $\sqrt{(c+3-\zeta)^2 + 8[\zeta(1-\tfrac{\mu}{2\alpha \lambda}) -1-c] } < \zeta +1-c$, which is always satisfied. Hence, $\beta_{+} \in (0,1)$ if $ \zeta \ge c+3$ and $ \zeta> \tfrac{2\alpha \lambda(1+c)}{2\alpha \lambda-\mu}$, or if $\scalemath{1}{\zeta< c+3}$. Combining this with the conditions for $\beta_{+} \in \mathbb{R}$, yields the regions of $\zeta$ for which $\beta_{+} \in (0,1)$. Note here that if $ \zeta> \tfrac{2\alpha \lambda(1+c)}{2\alpha \lambda-\mu}$, then $\beta_+ \in \mathbb{R}$. The condition on $c$ in a) results from the fact that, for the region to exist, the lower bound must have a lower value than the upper bound. Likewise, we consider $\beta_-$. Observe that $\beta_- > 0$ iff $\zeta < \min \{c+3, \tfrac{2\alpha \lambda(c+1)}{2\alpha \lambda- \mu} \}$ and $\beta_- <1$ iff $\sqrt{(c+3-\zeta)^2 + 8[\zeta(1-\tfrac{\mu}{2\alpha \lambda}) -1-c] } > c-\zeta-1$, which is satisfied. Combining the above with the conditions ensuring that $\beta_{-} \in \mathbb{R}$, gives the regions of $\zeta$ for which $\beta_{-} \in (0,1)$.

Then, we study local stability. First, we consider the DFE $(0,0)$. Linearising \eqref{simplesystem} about $(0,0)$, we find that the eigenvalues of the Jacobian matrix are $ 2\alpha \lambda- \mu$ and $ -(c+1)<0 $, where the former is negative iff $2\alpha \lambda < \mu$. Thus, the equilibrium  is locally exponentially stable (LES) if $2\alpha \lambda < \mu$, and a saddle point if $2\alpha \lambda > \mu$. The case $2\alpha \lambda = \mu$ is studied separately, yielding asymptotic stability. Computations are omitted due to space constraints. 
Next, we consider the DFE $(1,0)$. In a similar way, we linearize the system about the equilibrium and we find the eigenvalues of the Jacobian matrix, which are given by  $ c-1> 0 $ and $- \mu<0$, implying that the DFE $(1,0)$ is a saddle point.
Now consider the EE $(0,1-\tfrac{\mu}{2\alpha \lambda})$, with  $\mu < 2\alpha \lambda$, so it exists. Linearising the system about this equilibrium, we obtain a Jacobian matrix with eigenvalues $ \zeta(1-\tfrac{\mu}{2\alpha \lambda} ) - 1 -c$ and $\mu - 2\alpha \lambda<0$, so the equilibrium is LES if $\zeta< \tfrac{2\alpha \lambda(1+c)}{2\alpha \lambda-\mu}$, and a saddle point if the opposite inequality holds.  If the equality holds, \eqref{simplesystem} is studied directly (computations omitted), yielding asymptotic stability. 
Finally, consider the interior EE with $x= \beta_{+}$. Let $\beta_+ \in (0,1)$ and $\mu < 2\alpha \lambda (1-\beta_+)$ for existence. 
Linearising \eqref{simplesystem} about the EE yields the Jacobian matrix
\begin{equation*}
A=\begin{bmatrix} 2 \beta_+(1-\beta_+)  & \zeta \beta_+ (1-\beta_+) \\ -\tfrac{\mu}{1-\beta_+}(1-\tfrac{\mu}{2\alpha \lambda(1-\beta_+)}  ) &\mu-2\alpha \lambda (1-\beta_+) \end{bmatrix}.
\end{equation*}
By the determinant-trace method, the EE with $x= \beta_{+}$ is LES if 
$\tfrac{4 \alpha \lambda}{\zeta} (1-\beta_+)^2 < \mu < 2(1-\beta_+)(\alpha \lambda- \beta_+)$ and unstable if at least one of the opposite inequalities holds. By replacing $\beta_+$ with $\beta_-$ in the argument above, we complete the proof by obtaining the conditions for the other interior EE.

\subsection{Proof of Proposition~\ref{prop2}}\label{app:proofprop2}

Item i) follows from Lemma \ref{propequi}. We now prove ii). Under the parameter constraints imposed by the hypothesis of the proposition, it follows from Lemma \ref{propequi} that \eqref{simplesystem} has three equilibria on the boundary: $(0,0)$, $(1,0)$ and $(0,1-\tfrac{\mu}{2\alpha \lambda })$, all of which are saddle points. Moreover, the interior EE with $x=\beta_{-}$ does not exist. For existence of the interior EE with $x=\beta_{+}$, we need $ \lambda> \tfrac{\mu}{2\alpha(1-\beta_{+})}$, which is equivalent to  $\scalemath{0.98}{\sqrt{(c+3-\zeta)^2 + 8[\zeta(1-\tfrac{\mu}{2\alpha \lambda}) -1-c] }< \zeta+1-c-\tfrac{2\mu}{\alpha \lambda}}$, where the right-hand side (RHS) is positive for $ \lambda> \tfrac{\mu}{2\alpha}$ and $\zeta> c+3$. Note that $c+3 \le \tfrac{2 \alpha \lambda(1+c)}{2\alpha \lambda-\mu}$ iff $c \ge \tfrac{4 \alpha \lambda}{\mu}-3$. Squaring both sides, algebraic simplifications yield the equivalent expression $\scalemath{1}{\tfrac{\mu}{\alpha \lambda} + c-1  > 0}$, which holds.

The interior EE is LES iff $\tfrac{4 \alpha \lambda}{\zeta} (1-\beta_+)^2 < \mu < 2(1-\beta_+)(\alpha \lambda- \beta_+)$ (Lemma \ref{propequi}). The condition $\mu < 2(1-\beta_+)(\alpha \lambda- \beta_+)$ is equivalent to 
$\scalemath{.85}{        \tfrac{1}{2}[\alpha \lambda +1  + \tfrac{1}{2} (\zeta-c-3)] \sqrt{(c+3-\zeta)^2 + 8\left[\zeta\left(1-\tfrac{\mu}{2\alpha \lambda}\right) -1-c\right] }}<$
$        \scalemath{.85}{2\alpha  \lambda - \mu + \tfrac{1}{4}(\zeta-c-3)^2+ \tfrac{1}{2}(\zeta-c-3)(\alpha \lambda +1) + [\zeta(1-\tfrac{\mu}{2\alpha \lambda})-1}$
        $\scalemath{.85}{-c].}
$
For  $ \lambda> \tfrac{\mu}{2\alpha}$ and $\zeta> \max \{c+3,\tfrac{2\alpha \lambda(1+c)}{2\alpha \lambda-\mu} \}$, both the left-hand side (LHS) and RHS of the equation above are positive. After squaring both sides and some straightforward rewriting, we obtain the equivalent condition $(2\alpha \lambda - \mu)\zeta^2 -2(\alpha\lambda)^2 [c-3 + \tfrac{1}{\alpha \lambda}(c+1+2 \mu) ]\zeta +2(\alpha\lambda)^2 (c^2 + 2\alpha \lambda(1-c) -2 \mu -1) <0$, where the roots of the polynomial are given by $\scalemath{0.93}{\bar{\zeta}_{\pm} = \tfrac{\alpha \lambda}{2 \alpha \lambda- \mu} ((c+1) [1 \pm \sqrt{(\alpha \lambda -1)^2 +2 \mu}] + \alpha \lambda(c-3) +2 \mu )}$. The leading term of the polynomial on the LHS is positive, so the condition is satisfied iff $\bar{\zeta}_{-} < \zeta < \bar{\zeta}_{+}$. 

Next, $\mu > \tfrac{4 \alpha \lambda}{\zeta} (1-\beta_+)^2$ is equivalent to
 $\scalemath{.85}{ (\zeta+1-c)  \sqrt{(c+3-\zeta)^2 + 8[\zeta(1-\tfrac{\mu}{2\alpha \lambda}) -1-c] }
       > \zeta^2  +(c-1)^2}$ $ \scalemath{.85}{+2\zeta(1-c-\tfrac{5\mu}{4\alpha \lambda}) },$
 where the roots of the RHS are  $\tilde{\zeta}_{\pm}=c - 1 + \tfrac{5 \mu}{4 \alpha \lambda} \pm \sqrt{\tfrac{5 \mu}{2\alpha \lambda}(c-1+ \tfrac{5 \mu}{8 \alpha \lambda})}$. For $c<\tfrac{32 \alpha \lambda}{5 \mu}-3$, we have $c+3> \tilde{\zeta}_{+} $, so the RHS is positive. Squaring both sides and rewriting yields the equivalent expression  $- \zeta^3 +[2(c-1) + \tfrac{25\mu}{4\alpha \lambda}] \zeta^2 - (c -1)^2 \zeta < 0$, where the roots of the LHS are $\zeta=0$ and
 $ \breve{\zeta}_{\pm} = c-1 + \tfrac{25\mu}{8\alpha \lambda}  \pm \frac{5}{2} \sqrt{\tfrac{\mu}{\alpha \lambda} (c-1+ \tfrac{25\mu}{16\alpha \lambda})}.$ Since $c>1$ implies $  \breve{\zeta}_{-}< c-1$, we need $\zeta >  \breve{\zeta}_{+}$ for stability.

\subsection{Proof of Theorem~\ref{theo1}}\label{app:prooftheo1}

Under the conditions of Theorem~\ref{theo1}, \eqref{simplesystem} has three saddle points on the boundary: $(0,0)$, $(1,0)$ and $(0,1-\tfrac{\mu}{2\alpha \lambda })$, and the unique fully unstable interior EE $(\beta_{+},1-\tfrac{\mu}{2\alpha \lambda(1-\beta_+)})$. 
Also, $\dot{y} < 0$ if $y\neq 0$ and $\tfrac{\mu}{2\alpha \lambda}>(1-x)(1-y)$. For all $y>1-\tfrac{\mu}{2\alpha \lambda}$, we have $\tfrac{\mu}{2\alpha \lambda}>1-y \ge (1-x)(1-y)$, so $\dot{y} < 0$. Furthermore, $\dot{y} = -\mu xy \le 0$ at $y=1- \tfrac{\mu}{2\alpha \lambda}$, which implies that the region $[0,1] \times [0,1-\tfrac{\mu}{2\alpha \lambda}]$ is attractive and invariant.

We study now the behaviour of the system near the boundaries of $[0,1] \times [0,1-\tfrac{\mu}{2\alpha \lambda}]$, and examine whether it is possible to reach the boundary of the domain. Consider the boundary $x=1$. Let us assume that the trajectory reaches $x=1-\varepsilon$, with $\varepsilon>0$ arbitrarily small at time $t_0$. From \eqref{simplesystem}, we observe that $\dot x(t_0)=\varepsilon(1-\varepsilon)(1+\zeta y-c-2\varepsilon)$ when $x(t_0)=1-\varepsilon$. Hence, $\dot{x}$ can be positive only if $y > \tfrac{c-1}{ \zeta}$. We consider the regions $\mathcal R_\varepsilon:=[1-\varepsilon,1]\times[\tfrac{c-1}{\zeta},1-\tfrac{\mu}{2\alpha \lambda}]$ and $\scalemath{1}{\mathcal S_\varepsilon:=[1-\varepsilon,1]\times[0,\tfrac{c-1}{\zeta}]}$, where $\dot{x}$ can be positive only in $\mathcal R_\varepsilon$, while it is negative in $\mathcal S_\varepsilon$. 
Furthermore, we derive the following uniform bound for any pair $(x,y)\in\mathcal R_\varepsilon$: $\dot y=-\mu y + 2\alpha \lambda y\varepsilon(1-y)<-\mu\tfrac{c-1}{\zeta}<0$, for $\varepsilon$ sufficiently small. Similarly, we observe that $\dot{x} = x(1-x) (2x+\zeta y -1-c) \le k(1-x)$, for some constant $k>0$. 
These bounds yield a strict bound on the distance between the trajectory and $x=1$ before the trajectory exits the region $\mathcal R_\varepsilon$ from the bottom and enters $\mathcal S_\varepsilon$. 
We define $u(t)=1-x(t)$. The uniform bound on $\dot x$ in $\mathcal R_\varepsilon$ is equivalent to $-\dot{u} \le - k(-u)$, so by the Gronwall-Bellman  inequality~\cite{pachpatte1997inequalities}, $-u(t) \le -u(t_0)e^{-k(t-t_0)}$, which is equivalent to $x(t) \le 1-\varepsilon e^{-k (t-t_0)}$, for any $t\geq t_0$. Here, we used the fact that $x(t_0)=1-\varepsilon$. 
Then, the uniform bound on $\dot y$ in $\mathcal R_\varepsilon$ is used to derive a bound on the time needed for the trajectory to exit $\mathcal R_\varepsilon$. Specifically, since the length along the $y$-axis of $\mathcal R_\varepsilon$ is equal to $(1-\tfrac{\mu}{2\alpha \lambda}-\tfrac{c-1}{\zeta})$, and the time-derivative of the trajectory along the $y$-component is negative and greater in modulus than $\mu\tfrac{c-1}{\zeta}$, then there necessarily exists a time $\tilde{t} \leq \zeta(1-\tfrac{\mu}{2\alpha \lambda}-\tfrac{c-1}{\zeta})/(\mu (c-1))$ such that $y(t_0+\tilde t)<\tfrac{c-1}{\zeta}$ and $x(t_0+\tilde t) \le 1-\varepsilon'$, with $\varepsilon'=\varepsilon e^{-k\tilde t}$. 
Hence, the trajectory will exit from $\mathcal R_\varepsilon$ and will enter $\mathcal S_\varepsilon$, in which $\dot x<0$ and $\dot y<0$. This establishes that the trajectory cannot further approach the boundary $x=1$, nor re-enter  $\mathcal R_\varepsilon$ from the boundary between $\mathcal S_\varepsilon$ and $\mathcal R_\varepsilon$. Thus, there exists a constant $\varepsilon'>0$ such that $[0,1-\varepsilon']\times [0,1-\tfrac{\mu}{2\alpha \lambda}]$ is positively invariant for \eqref{simplesystem}.

A similar argument guarantees that any trajectory that starts from the interior is bounded away from the boundaries $y=0$ and $x=0$. Since convergence to the boundaries is impossible, the boundary equilibria points cannot be reached if the initial conditions of the system are in the interior of the domain $(0,1) \times (0,1)$.
Finally, we consider the open set $(x,y) \in (0,1) \times (0,1-\tfrac{\mu}{2\alpha \lambda})$. Since the unique interior EE is an unstable point under the above conditions, there does not exist a homoclinic orbit. It follows directly from the generalised Poincar\'{e}-Bendixson theorem~\cite{teschl2012ordinary} that every non-empty compact $\omega$-limit set of an orbit is periodic.


\begin{thebibliography}{29}
\providecommand{\newblock}{\relax}
\providecommand{\bibinfo}[2]{#2}
\providecommand{\BIBentrySTDinterwordspacing}{\spaceskip=0pt\relax}
\providecommand{\BIBentryALTinterwordstretchfactor}{4}
\providecommand{\BIBentryALTinterwordspacing}{\spaceskip=\fontdimen2\font plus
\BIBentryALTinterwordstretchfactor\fontdimen3\font minus
  \fontdimen4\font\relax}
\providecommand{\BIBforeignlanguage}[2]{{%
\expandafter\ifx\csname l@#1\endcsname\relax
\typeout{** WARNING: IEEEtran.bst: No hyphenation pattern has been}%
\typeout{** loaded for the language `#1'. Using the pattern for}%
\typeout{** the default language instead.}%
\else
\language=\csname l@#1\endcsname
\fi
#2}}
\providecommand{\BIBdecl}{\relax}
\BIBdecl

\bibitem{Nowzari2016}
C.~Nowzari, V.~M. Preciado, and G.~J. Pappas, ``Analysis and control of
  epidemics: a survey of spreading processes on complex networks,'' \emph{IEEE
  Control Syst. Mag.}, vol.~36, no.~1, pp. 26--46, 2016.

\bibitem{Mei2017}
W.~Mei, S.~Mohagheghi, S.~Zampieri, and F.~Bullo, ``On the dynamics of
  deterministic epidemic propagation over networks,'' \emph{Annu. Rev. Contr.},
  vol.~44, pp. 116--128, 2017.

\bibitem{Pare2020}
P.~E. Paré, C.~L. Beck, and T.~Başar, ``Modeling, estimation, and analysis of
  epidemics over networks: An overview,'' \emph{Annu. Rev. Control}, vol.~50,
  pp. 345--360, 2020.

\bibitem{zinoreview}
L.~Zino and M.~Cao, ``Analysis, prediction, and control of epidemics: A survey
  from scalar to dynamic network models,'' \emph{IEEE Circuits Syst. Mag.},
  vol.~21, no.~4, pp. 4--23, 2021.

\bibitem{giordano2020modelling}
G.~Giordano \emph{et~al.}, ``{Modelling the {COVID}-19 epidemic and
  implementation of population-wide interventions in Italy},'' \emph{Nat.
  Med.}, vol.~26, no.~6, pp. 855--860, 2020.

\bibitem{DellaRossa2020}
F.~Della~Rossa \emph{et~al.}, ``A network model of {I}taly shows that
  intermittent regional strategies can alleviate the {COVID}-19 epidemic,''
  \emph{Nat. Comm.}, vol.~11, no.~1, p. 5106, 2020.
  
  \bibitem{Funk2010}
S.~Funk, M.~Salath\'{e}, and V.~A. Jansen, ``{Modelling the influence of human
  behaviour on the spread of infectious diseases: a review.}'' \emph{J. R. Soc.
  Interface}, vol.~7, no.~50, pp. 1247--1256, 2010.

\bibitem{Wang2015}
Z.~Wang, M.~A. Andrews, Z.-X. Wu, L.~Wang, and C.~T. Bauch, ``Coupled
  disease–behavior dynamics on complex networks: A review,'' \emph{Phys. Life
  Rev.}, vol.~15, pp. 1 -- 29, 2015.

\bibitem{Sahneh2012}
F.~D. Sahneh, F.~N. Chowdhury, and C.~M. Scoglio, ``On the existence of a
  threshold for preventive behavioral responses to suppress epidemic
  spreading,'' \emph{Sci. Rep.}, vol.~2, no.~1, p. 632, Sep 2012.

\bibitem{Granell2013}
C.~Granell, S.~G\'{o}mez, and A.~Arenas, ``{Dynamical interplay between
  awareness and epidemic spreading in multiplex networks},'' \emph{Phys. Rev.
  Lett.}, vol. 111, no.~12, p. 128701, 2013.

\bibitem{9089218}
L.~Zino, A.~Rizzo, and M.~Porfiri, ``On assessing control actions for epidemic
  models on temporal networks,'' \emph{IEEE Control Syst. Lett.}, vol.~4,
  no.~4, pp. 797--802, 2020.

\bibitem{frieswijk2021time}
K.~Frieswijk, L.~Zino, and M.~Cao, ``A time-varying network model for sexually
  transmitted infections accounting for behavior and control actions,''
  \emph{Int. J. Robust Nonlinear Control}, 2021.

\bibitem{Peng2021}
K.~Peng \emph{et~al.}, ``A multilayer network model of the coevolution of the
  spread of a disease and competing opinions,'' \emph{Math. Models Methods
  Appl. Sci.}, vol.~31, no.~12, pp. 2455--94, 2021.

\bibitem{She2022}
B.~She, J.~Liu, S.~Sundaram, and P.~E. Par{\'e}, ``On a networked {SIS}
  epidemic model with cooperative and antagonistic opinion dynamics,''
  \emph{IEEE Trans. Control Netw. Syst.}, pp. 1--1, 2022.

\bibitem{Huang2022}
Y.~Huang and Q.~Zhu, ``Game-theoretic frameworks for epidemic spreading and
  human decision-making: A review,'' \emph{Dyn. Games Appl.}, 2022.

\bibitem{Hota2019}
A.~R. Hota and S.~Sundaram, ``Game-theoretic vaccination against networked
  {SIS} epidemics and impacts of human decision-making,'' \emph{IEEE Trans.
  Control. Netw. Syst.}, vol.~6, no.~4, pp. 1461--1472, 2019.

\bibitem{Khazaei2021}
H.~Khazaei, K.~Paarporn, A.~Garcia, and C.~Eksin, ``Disease spread coupled with
  evolutionary social distancing dynamics can lead to growing oscillations,''
  in \emph{60th IEEE Conf. Dec. Control}, 2021, pp. 4280--4286.

\bibitem{Elokda2021}
E.~Elokda, S.~Bolognani, and A.~R. Hota, ``A dynamic population model of
  strategic interaction and migration under epidemic risk,'' in \emph{60th IEEE
  Conf. Dec. Control}, 2021, pp. 2085--2091.

\bibitem{Martins2022}
N.~C. Martins, J.~Certorio, and R.~J. La, ``Epidemic population games and
  evolutionary dynamics,'' \emph{preprint, arXiv:2201.10529}, 2022.

\bibitem{Satapathi2022}
A.~H. A.~Satapathi, N.K.~Dhar and V.~Srivastava, ``Epidemic propagation under
  evolutionary behavioral dynamics: Stability and bifurcation analysis,''
  \emph{preprint, arXiv:2203.10276}, 2022.

\bibitem{ye2021game}
M.~Ye, L.~Zino, A.~Rizzo, and M.~Cao, ``Game-theoretic modeling of collective
  decision making during epidemics,'' \emph{Phys. Rev. E}, vol. 104, no.~2, p.
  024314, 2021.

\bibitem{VanMieghem2009}
P.~V. Mieghem, J.~Omic, and R.~Kooij, ``Virus spread in networks,''
  \emph{IEEE/ACM Trans. Netw.}, vol.~17, no.~1, pp. 1--14, 2009.

\bibitem{Zino2016}
L.~Zino, A.~Rizzo, and M.~Porfiri, ``{Continuous-time discrete-distribution
  theory for activity-driven networks},'' \emph{Phys. Rev. Lett.}, vol. 117,
  2016.

\bibitem{hofbauer1998evolutionary}
J.~Hofbauer, K.~Sigmund \emph{et~al.}, \emph{Evolutionary games and population
  dynamics}.\hskip 1em plus 0.5em minus 0.4em\relax Cambridge University Press,
  1998.

\bibitem{Como2021}
G.~Como, F.~Fagnani, and L.~Zino, ``Imitation dynamics in population games on
  community networks,'' \emph{IEEE Trans. Control. Netw. Syst.}, vol.~8, no.~1,
  pp. 65--76, 2021.

\bibitem{Zino2017}
L.~Zino, A.~Rizzo, and M.~Porfiri, ``{An analytical framework for the study of
  epidemic models on activity driven networks},'' \emph{J. Complex Netw.},
  vol.~5, no.~6, pp. 924--952, 2017.

\bibitem{limitmarkov}
T.~G. Kurtz, ``{Solutions of Ordinary Differential Equations as Limits of Pure
  Jump Markov Processes},'' \emph{J. Appl. Probab.}, vol.~7, pp. 49--58, 1970.

\bibitem{Blanchini1999}
F.~Blanchini, ``Set invariance in control,'' \emph{Automatica}, vol.~35,
  no.~11, pp. 1747--1767, 1999.

\bibitem{pachpatte1997inequalities}
B.~G. Pachpatte, \emph{Inequalities for differential and integral
  equations}.\hskip 1em plus 0.5em minus 0.4em\relax Elsevier, 1997.

\bibitem{teschl2012ordinary}
G.~Teschl, \emph{Ordinary Differential Equations and Dynamical Systems}.\hskip
  1em plus 0.5em minus 0.4em\relax Providence, RI: American Mathematical
  Society, 2012, vol. 140.
  

\end{thebibliography}
\end{document}